\documentclass[12pt,notitlepage,a4paper,reqno,ps]{amsart}
\usepackage{eurosym}
\usepackage{amssymb}
\usepackage{amstext}
\usepackage{color}
\usepackage{enumitem}

\newcommand{\N}{\mathbb{N}}
\newcommand{\R}{\mathbb{R}}
\newcommand{\loc}{\textnormal{loc}}
\newcommand{\medint}{-\kern  -,375cm\int}

\newenvironment{michelarev}{\color{blue}}{\color{black}}
\newcommand{\bmicr}{\begin{michelarev}}
\newcommand{\emicr}{\end{michelarev}}

\newenvironment{samuelerev}{\color{blue}}{\color{black}}
\newcommand{\bsamr}{\begin{samuelerev}}
\newcommand{\esamr}{\end{samuelerev}}

\newenvironment{giacomorev}{\color{green}}{\color{black}}
\newcommand{\bgiar}{\begin{giacomorev}}
\newcommand{\egiar}{\end{giacomorev}}

\theoremstyle{plain}
\newtheorem{theorem}{Theorem}[section]
\newtheorem{corollary}[theorem]{Corollary}
\newtheorem{lemma}[theorem]{Lemma}

\newtheorem{remark}[theorem]{Remark}
\newtheorem*{rmk*}{Remark}
\theoremstyle{definition}

\theoremstyle{remark}

\numberwithin{equation}{section} \makeatletter

\renewcommand{\p@enumi}{\thesection.}
\makeatother  \allowdisplaybreaks

\begin{document}
\title{Regularity for obstacle problems \\ without structure conditions}
\author[G. Bertazzoni, S. Ricc\`o]{Giacomo Bertazzoni, Samuele Ricc\`o}
\address{Dipartimento di Scienze Fisiche, Informatiche e Matematiche, via Campi 213/b, 41125 Mo-dena, Italy}


\begin{abstract}

This paper deals with the Lipschitz regularity of minimizers for a class of variational obstacle problems with possible occurance of the Lavrentiev phenomenon. In order to overcome this problem, the availment of the notions of relaxed functional and Lavrentiev gap are needed. The main tool used here is a ingenious Lemma which reveals to be crucial because it allows us to move from the variational obstacle problem to the relaxed-functional-related one. This is fundamental in order to find the solutions' regularity that we intended to study. We assume the same Sobolev regularity both for the gradient of the obstacle and for the coefficients.
\end{abstract}

\maketitle

\begin{center}
\fbox{\today}
\end{center}


\section{Introduction}
\label{intro}

In this manuscript we study the Lipschitz continuity of the solutions to variational obstacle problems of the form
\begin{equation}
\label{obst-def0}
\min \, \left\{ \int_\Omega f(x,Dw) : w \in \mathcal{K}_\psi(\Omega) \right\}
\end{equation}
in the case of $p,q-$growth condition, where $\Omega$, $f$ and $\mathcal{K}_\psi(\Omega)$ will be specified below and where we assume that Lavrentiev phenomenon may occur.
\\
The relationship between the ellipticity and the growth exponent we impose, namely \eqref{coeff_eq}, is the one considered for the first time in the series of papers \cite{EMM16_a}, \cite{EMM16_b}, \cite{EMM19}, \cite{EMM3} and it is sharp (in view of the well known counterexamples, see for instance \cite{M87}) also to obtain the Lipschitz continuity of solutions to elliptic equations and systems and minimizers of related functionals with $p,q-$growth. Regularity results under non standard growth conditions, a research branch started after the pioneering papers by Marcellini \cite{M89}, \cite{M91}, \cite{M93}, has recently attracted growing attention, see among the others \cite{BS19}, \cite{BDM13}, \cite{CM15-1},
\cite{DFP19}, \cite{HH19}, \cite{M19}, \cite{M20}, \cite{RT19}.
\\
The term \textit{Lavrentiev phenomenon} refers to a surprising result first demonstrated in 1926 by M. Lavrentiev in \cite{L26}. There it was shown that it is possible for the variational integral of a two-point Lagrange problem, which is sequentially weakly lower semicontinuous on the admissible class of absolutely continuous functions, to possess an infimum on the dense subclass of $\mathcal{C}^1$ admissible functions that is \textit{strictly greater} than its minimum value on the full admissible class. Since the paper of Lavrentiev indeed, many contributions appeared in this direction, see \cite{BB}, \cite{BM}, \cite{ELM04}, \cite{ELP}, \cite{Zhikov}.
\\
In our case, it is still possible to have occurrence of Lavrentiev phenomenon due to the nonstandard growth conditions required on the lagrangian.  In this respect, under our assumptions, the Lavrentiev phenomenon can be reformulated in these terms:
\begin{equation}
\label{def-lavr}
\inf_{w \in (W^{1,p} \, \cap \, \{ w \ge \psi \})} \int_\Omega f(x, Dw) \, dx \, < \, \inf_{w \in (W^{1,q} \, \cap \, \{ w \ge \psi \})} \int_\Omega f(x, Dw) \, dx
\end{equation}
So, the aim of this work is to complement the results cointained in the paper \cite{CEP19}, where authors obtain Lipschitz regularity results for obstacle problems with Sobolev regularity for the coefficients and where the lagrangian $f$ satisfies $p,q-$growth conditions without assuming that Lavrentiev phenomenon may occur. 
\\
This phenomenon is a clear obstruction to regularity, since \eqref{def-lavr} prevents minimizers to belong to $W^{1,q}$. Notice that \eqref{def-lavr} cannot happen if $p = q$ or if $f$ is autonomous (it not depends on variable $x$) and convex. Moreover, as pointed out in Section 3 of \cite{ELM04}, the appearance of \eqref{def-lavr} has geometrical reasons and cannot be spotted in a direct way by standard elliptic regularity techniques. Therefore, the basic strategy in getting regularity results consists in excluding the occurrence of \eqref{def-lavr} by imposing that the Lavrentiev gap functional $\mathcal{L}(u)$, defined in \eqref{gap}, vanishes on solutions.
\\
However, here in this manuscript we adopt a different viewpoint, following the lines of \cite{EMM3}. We present a general Lipschitz regularity result by covering the case in which the Lavrentiev phenomenon may occur. In this respect, a key role will be played by the relaxed functional. We therefore need to introduce the exact framework of relaxation in the case of obstacle problems and then we will state our main result. The crucial step will be constituted by Lemma \ref{conv-funz} which is the natural counterpart of the necessary and sufficient condition to get the absence of Lavrentiev phenomenon.
\\
We will state and prove the result with Sobolev dependence on both the obstacle and the partial map $x \mapsto D_{\xi} f(x,\xi)$. A model functional that is covered by our results is
\begin{equation*}
w \mapsto \int_\Omega \left[ |Dw|^p + a(x) \, (1 + |Dw|^2)^\frac{q}{2} \right] \, dx
\end{equation*}
with $q > p > 1$ and $a(\cdot)$ a bounded Sobolev coefficient.
\\
The plan of the paper is the following: in Section \ref{state} we state our model problem and the main results of the paper, in Section \ref{prelim} we present some preliminary results we need in the sequel; Section \ref{a-priori} is devoted to the presentation of our a-priori estimate and finally in Section \ref{approx-lavr} we present the Lipschitz regularity results for solutions to the relaxed obstacle problem. 


\section{Statement of the problem and of the main results} 

\label{state}

Let us present in details the setting of our problem.
\\
Here $\Omega$ is a bounded open set of $\mathbb{R}^n$, $n \ge 2$. We will deal with variational integral
\begin{equation}
\label{main-F}
\mathcal{F}(u) := \int_\Omega f(x, Du) \, dx
\end{equation}
where $f: \Omega \times \mathbb{R}^n \rightarrow [0, + \infty)$ is a Carath\'eodory function which is convex and of class $\mathcal{C}^2$ with respect to the second variable. We consider $p$ and $q$ bounded by
\begin{equation}
\label{coeff_eq}
1 \le \frac{q}{p} < 1 + \frac{r - n}{r \, n} = 1 + \frac{1}{n} - \frac{1}{r}
\end{equation}
where $r>n$, so $\frac{1}{r} < \frac{1}{n}$ and where we consider $q > p \ge \, 2$. 
We suppose that there exist two positive constants $\nu$, $L$ and a function $h: \Omega \rightarrow [0, + \infty)$ such as $h(x) \in L^r_\loc(\Omega)$ such that
\begin{equation}
\label{HpF1}
\nu \, (1 + |\xi|^2)^\frac{p}{2} \le f(x, \xi) \le L \, (1 + |\xi|^2)^\frac{q}{2} \\[3.7mm]
\end{equation}
\begin{equation}
\label{HpF2}
\nu \, (1 + |\xi|^2)^\frac{p - 2}{2} \, |\lambda|^2 \le \sum_{i,j} f_{\xi_i \xi_j}(x, \xi) \, \lambda_i \, \lambda_j \le L \, (1 + |\xi|^2)^\frac{q-2}{2} \, |\lambda|^2  \\
\end{equation}
\begin{equation}
\label{HpF3}
|f_{x \xi}(x, \xi)| \le h(x) \, (1 +|\xi|^{2})^\frac{q-1}{2}
\end{equation}
\\
for all $\lambda, \xi \in \mathbb{R}^n$, $\lambda = \lambda_i$, $\xi = \xi_i$, $i = 1, 2, \dots, n$, a.e. in $\Omega$.
\\
\\
Furthermore, we give meaning of our obstacle problem \eqref{obst-def0}. Namely, the function $\psi : \, \Omega \rightarrow [- \infty, + \infty)$, called \textit{obstacle}, belongs to the Sobolev space $W^{1, p}(\Omega)$ and the class $\mathcal{K}_\psi (\Omega)$ is defined as follows
\begin{equation}
\label{classeA}
\mathcal{K}_\psi(\Omega) := \left\{ w \in u_0 + W^{1,p}_0(\Omega) : w \ge \psi \,\, \textnormal{a.e. in $\Omega$} \right\}
\end{equation}
where $u_0$ is a fixed boundary value. In order to prove Theorem \ref{lip-solution} we will need to assume $u_0 \in W^{1,q}(\Omega)$. 

\noindent
To avoid trivialities, in what follows we shall assume that $\mathcal{K}_\psi$ is not empty. We also assume that a solution to \eqref{obst-def0} is such that $f(x, Du) \in L^1_\loc(\Omega)$. As it has been shown in \cite{EPdN21}, in case of non-standard growth condition (at least in the autonomous case), this turns to be the right class of competitors.
\\
\begin{remark}
\label{oss-classe-vuota}
Let us notice that, by replacing $u_0$ by $\tilde{u}_0 = \max \, \{u_0, \psi\}$, we may assume that the boundary value function $u_0$ satisfies $u_0 \ge \, \psi$ in $\Omega$. Indeed $\tilde{u}_0  = (\psi - u_0)^+ + u_0$ and since
\begin{equation*}
0 \le (\psi - u_0)^+ \le (u - u_0)^+ \in W^{1,1}_0(\Omega),
\end{equation*}
the function $(\psi \, - \, u_0)^+$, and hence $u \, - \, \tilde{u}_0$, belongs to $W^{1,1}_0(\Omega)$. Moreover assumptions $f(x, Du) \in L^1_\loc(\Omega)$ and $f(x, Du_0) \in L^1_\loc(\Omega)$ imply $f(x, D \tilde u_0) \in L^1_\loc(\Omega)$. Indeed we have
\begin{eqnarray*}
\int_\Omega f(x, D\tilde u_0) \, dx
&=&\int_{\Omega \, \cap \, \{u_0 \ge \psi\}} f(x, Du_0) \, dx + \int_{\Omega \, \cap \, \{u_0 < \psi \}} f(x, D\psi) \, dx \\
&\le& \int_\Omega \left[ f(x, Du_0) + f(x, D\psi) \right] \, dx \, < \, +\infty	
\end{eqnarray*}
where we used that $f(x, \xi) \ge 0$, by virtue of the left inequality in \eqref{HpF1}.
\end{remark}

Throughout the paper we will denote by $B_\rho$ and $B_R$ balls of radii respectively $\rho$ and $R$ (with $\rho < R$) compactly contained in $\Omega$ and with the same center, let us say $x_0 \in \Omega$. Moreover in the sequel, constants will be denoted by $C$, regardless their actual value. Only the relevant dependencies will be highlighted.
\\
\\
Our first main result is an a-priori estimate and reads as follows
\begin{theorem}[A priori estimate]
\label{main-result}
Let $u \in \mathcal{K}_\psi(\Omega)$ be a smooth solution to the obstacle problem \eqref{obst-def0} under the assumptions \eqref{HpF1}, \eqref{HpF2}, \eqref{HpF3} and \eqref{coeff_eq}. 
If $\psi \in W^{2, r}_\loc(\Omega)$, then $u \in W^{1, \infty}_\loc(\Omega)$ and the following estimate 
\begin{equation}
\label{a-priori}
\|Du\|_{L^\infty(B_\rho)} \le C \, \left\{ \int_{B_R} [1 + f(x, Du)] \, dx \right\}^\beta
\end{equation}
holds for every $0 < \rho < R$ and with positive constants $C$ and $\beta$ depending on $n, r, p, q, \nu, L, R, \rho$ and on the local bounds for $\|D \psi\|_{W^{1, r}}$ and $\|h\|_{L^r}$. 
\end{theorem}

Now we want to present a meaningful definition of relaxation for problem \eqref{obst-def0} in the spirit of \cite{ABF03}, \cite{EMM3}, \cite{ELM04}, \cite{M86}. Inspired by \cite{DF19} and \cite{G21}, we consider the set
\begin{equation*}
\mathcal{K}^*_\psi(\Omega) \, := \, \{ w \in u_0 + W^{1,q}_0(\Omega) : w \ge \psi \,\, \textnormal{a.e. in $\Omega$} \}
\end{equation*}
Then, in spirit of \cite{ABF03} and \cite{M86}, we introduce the relaxed functional
\begin{equation}
\label{relax}
\overline{\mathcal{F}}(u) \, := \, \inf_{\mathcal{C}(u)} \{ \liminf_{j \rightarrow +\infty} \mathcal{F}(u_j) \}
\end{equation}
where $\mathcal{F}$ is defined in \eqref{main-F} and where
\begin{equation}
\label{classe_seq}
\mathcal{C}(u) \, := \, \left\{ \{ u_j\}_{j \in \N} \subset \mathcal{K}^*_\psi(\Omega) : u_j \rightharpoonup u \,\, \textnormal{weakly in $W^{1,p}(\Omega)$} \right\}
\end{equation}
and consider the \textit{Lavrentiev gap}
\begin{equation}
\label{gap}
\mathcal{L}(u) :=
	\begin{cases}
		\overline{\mathcal{F}}(u) - \mathcal{F}(u) \qquad & \text{if $\, \mathcal{F}(u) < \infty$} \\
		0 \qquad & \text{if $\, \mathcal{F}(u) = \infty$}
	\end{cases}
\end{equation}

Then, the Lipschitz regularity result we are able to prove is the following. 

\begin{theorem}
\label{lip-solution}
Assume that $f$ satisfies the hypothesis \eqref{HpF1}, \eqref{HpF2}, \eqref{HpF3} and \eqref{coeff_eq}. The Dirichlet problem
\begin{equation*}
\min \, \{ \overline{\mathcal{F}}(u) : u \in \mathcal{K}_\psi(\Omega) \}
\end{equation*}
with $\overline{\mathcal{F}}$ defined as in \eqref{relax} and $u_0 \in W^{1, q}(\Omega)$, has at least one locally Lipschitz continuous solution.
\end{theorem}

\section{Preliminary results}

\label{prelim}

First of all, we state the following lemma which has important application in the so called hole-filling method. Its proof can be found for example in \cite[Lemma 6.1]{Giusti}.
\begin{lemma}
\label{holf}
Let $h : [\rho_0, R_0] \to \mathbb{R}$ be a non-negative bounded function and $0 < \vartheta < 1$, $A, B \ge 0$ and $\beta > 0$. Assume that
\begin{equation*}
h(s) \le \vartheta \, h(t) + \frac{A}{(t - s)^\beta} + B
\end{equation*}
for all $\rho_0 \le s < t \le R_0$. Then
\begin{equation*}
h(r) \le \frac{c \, A}{(R_0 - \rho_0)^\beta} + c \, B
\end{equation*}
where $c = c(\vartheta, \beta) > 0$.
\end{lemma}

We now present the higher differentiability result we need in the sequel. Our hypothesis implies the ones needed for this theorem to be true. The proof can be found in \cite{G-stime}.
\begin{theorem}
\label{teo-chiara}
Let $u \in \mathcal{K}_\psi(\Omega)$ be a solution to the obstacle problem
\begin{equation}
\label{dis-var-new}
\int_\Omega D_\xi f (x, Du) \cdot D(\varphi - u) \, dx \ge 0 
\end{equation}
under the assumptions \eqref{HpF1}, \eqref{HpF2}, \eqref{HpF3} for the exponents $2 \le p < q < n < r$ such that \eqref{coeff_eq} holds true. Then the following estimate
\begin{equation}
\label{stima-chiara}
\int_{B_{R/8}} |D(V_p(Du(x)))|^2 \, dx \le C \left( 1 + \| D\psi \|_{L^r(B_{R/2})} + \| D^2\psi \|_{L^r(B_{R/2})}+ \| Du \|_{L^p(B_{R/2})} \right)^\alpha
\end{equation}
holds for all balls $B_{R/8} \subset  B_{R/2} \subset B_R \Subset \Omega$, with $C = C(\nu, L, q, p, r, n, R)$ and $\alpha = \alpha(p, q, r, n)$, where we defined
\begin{equation}
\label{vupi}
V_p(\xi) \, := \, \xi \, \left( 1 + |\xi|^2 \right)^\frac{p-2}{4}
\end{equation}
\end{theorem}

\begin{remark}
One can easily check that, for $p \ge 2$, there exists an absolute constant $c$ such that
\begin{equation}
\label{stima_vupi}
|\xi|^p \, \le \, c \, |V_p(\xi)|^2
\end{equation}
\end{remark}
\begin{remark}
Thanks to the fact that our first main result Theorem \ref{main-result} is an a priori estimate, we can use indifferently the variational inequality \eqref{dis-var-new} and the obstacle problem \eqref{obst-def0} in the following.
\end{remark}

\noindent
Once estimate \eqref{stima-chiara} is estabilished, by applying a suitable approximation procedure (see \cite{G-stime} for the details) we get the following corollary.
\begin{corollary}
Let $u$ and $\psi$ be as above. Then the following implication holds
\begin{equation}
\label{differenziabilita}
D\psi \in W^{1, r}_\loc(\Omega) \quad \Longrightarrow \quad \left( 1 + |Du|^2 \right)^\frac{p-2}{4} Du \in W^{1, 2}_\loc(\Omega)
\end{equation}
\end{corollary}
\noindent


\section{A priori estimate}
\label{a-priori}

\noindent

\subsection{The linearization procedure}
\label{lineariz}

The linearization procedure is a process which goes back to \cite{F85} and later was refined in \cite{D87}, see also \cite{F90}, \cite{F94}, \cite{FG}. We will follow the lines of \cite{BC20}.
\\
\\
We consider a smooth function $h_\varepsilon : (0, \infty) \to [0, 1]$ such that $h'_\varepsilon(s) \le 0 $ for all $s\in (0, \infty)$ and 
\begin{equation*}
h_\varepsilon(s) = 
     \begin{cases}
       1 &\quad\text{for} \quad s \le \varepsilon \\
       0 &\quad\text{for} \quad s \ge 2 \, \varepsilon \\
     \end{cases}
\end{equation*}
Consider the function 
\begin{equation*}
\varphi = u + t\cdot \eta \cdot  h_\varepsilon (u - \psi)  
\end{equation*}
with $\eta \in C_0^1 (\Omega)$, $\eta \ge 0$ and $0 < t <<1$ as test function, in the following variational inequality 
\begin{equation}
\label{dis-var-new}
\int_\Omega D_\xi f (x, Du) \cdot D(\varphi - u) \, dx \ge 0 
\end{equation}
that holds true for all $\varphi \in W^{1,q}_\loc(\Omega)$, $\varphi \ge \psi$. We have
\begin{equation*}
\int_\Omega D_\xi f(x, Du) \cdot D(\eta \, h_\varepsilon (u-\psi)) \, dx \ge 0 \qquad \forall \, \eta \in C_0^1(\Omega)
\end{equation*}
Since 
\begin{equation*}
\eta \mapsto L(\eta)=\int_\Omega D_\xi f(x, Du) \cdot D(\eta \, h_\varepsilon (u-\psi)) \, dx
\end{equation*}
is a bounded positive linear functional, by Riesz Representation Theorem there exists a non-negative measure $\lambda_\varepsilon$ such that 
\begin{equation*}
\int_\Omega D_\xi f(x, Du) \cdot D(\eta \, h_\varepsilon (u-\psi)) \, dx = \int_\Omega \eta \, d\lambda_\varepsilon  \qquad \forall \, \eta \in C_0^1(\Omega)
\end{equation*}
It is not difficult to prove that the measure $\lambda_\varepsilon$ is independent to $\varepsilon$. Therefore we can write
\begin{equation*}
\int_\Omega D_\xi f(x, Du) \cdot D(\eta \, h_\varepsilon (u-\psi)) \, dx = \int_\Omega \eta \, d\lambda \qquad \forall \, \eta \in C_0^1(\Omega)
\end{equation*}
By Theorem \ref{teo-chiara}, we have that
\begin{equation}
\label{def_V}
V_p(Du) := (1 + |Du|^2)^\frac{p - 2}{4} \, Du \in W^{1,2}_\loc (\Omega)
\end{equation}
Now, in order to identify the measure $\lambda$, we may pass to the limit as $\varepsilon \downarrow 0 $
\begin{equation}
\label{eq-g}
\int_\Omega -\text{div} (D_\xi f(x, Du)) \,\chi_{\left[u = \psi \right]} \, \eta \, dx = \int_\Omega \eta \, d\lambda \qquad \forall  \, \eta \in C_0^1(\Omega)
\end{equation} 
By introducing
\begin{equation}
\label{def-g}
g := -\text{div} (D_\xi f(x, Du)) \chi_{\left[u = \psi \right]}
\end{equation}
and combining our results we obtain
\begin{equation}
\label{eq:g}
\int_\Omega D_\xi f(x, Du) \cdot D\eta \, dx = \int_\Omega g \, \eta \, dx \qquad \forall  \, \eta \in C_0^1(\Omega)
\end{equation}
We are left to obtain an $L^r$ estimate for $g$: since $Du=D\psi$ a.e. on the contact set, by \eqref{HpF2} and \eqref{HpF3} and the assumption $D\psi \in W^{1,r}_\loc (\Omega; \, \R^n)$, we have 
\begin{eqnarray*}
|g|
&=& \left|\text{div} (D_\xi f(x, Du)) \, \chi_{\left[u = \psi \right]} \right| \\[3.5mm]
&=& \left|\text{div} (D_\xi f(x, D\psi)) \right| \\[2mm]
&\le& \sum_{k=1}^n |f_{\xi_k x_k}(x, D\psi)| + \sum_{k,i=1}^n |f_{\xi_k \xi_i}(x, D\psi) \, \psi_{x_k x_i}| \\ 
&\le& h(x) \, (1+|D\psi|^2)^{\frac{q-1}{2}}+ L \, (1+|D\psi|^2)^{\frac{q-2}{2}} \, |D^2 \psi|
\end{eqnarray*}
that is $g \in L^r_\loc (\Omega)$. 

\subsection{A priori estimate}
\label{estimate}

Our starting point is now \eqref{eq:g}. As long as we are proving an a priori estimate, we make use of the fact that $u \in W^{1,\infty}_{\rm loc}(\Omega)$, which is needed in order to let \eqref{eq:g} to be satisfied. By this further requirement and Theorem \ref{teo-chiara}, the ``second variation'' system holds
\begin{equation}
\label{second_variation}
\int_\Omega \left( \sum_{i,j=1}^n  f_{\xi_i \xi_j}(x, Du) \, u_{x_j x_s} \, D_{x_i} \, \varphi + \sum_{i=1}^n f_{\xi_i x_s}(x, Du) \, D_{x_i} \, \varphi \right) dx = \int_\Omega g \, D_{x_s} \, \varphi \, dx
\end{equation}
\\
for all $s = 1, \dots, n$ and for all $\varphi \in W^{1,2}_0(\Omega)$. We fix $0 < \rho < R$ with $B_R$ compactly contained in $\Omega$ and we choose $\eta \in \mathcal{C}^1_0(\Omega)$ such that $0 \le \eta \le 1$,  $\eta \equiv 1$ on $B_\rho$, $\eta \equiv 0$ outside $B_R$ and $|D \eta| \le \frac{C}{R - \rho}$. We test \eqref{second_variation} with
\begin{equation*}
\varphi = \eta^2 \, (1 + |Du|^2)^\gamma \, u_{x_s}
\end{equation*}
for some $\gamma \ge 0$ so that 
\begin{eqnarray*}
D_{x_i}\varphi
&=& 2 \, \eta \, \eta_{x_i} \, (1 + |Du|^2)^\gamma \, u_{x_s} \\[2.0mm]
&&+ 2 \, \eta^2 \, \gamma \, (1 + |Du|^2)^{\gamma-1} \, |Du| \, D_{x_i}(|Du|) \, u_{x_s} \\[2.0mm]
&&+ \eta^2 \, (1 + |Du|^2)^\gamma \, u_{x_s x_i}
\end{eqnarray*}
Inserting in \eqref{second_variation} we get: 
\begin{eqnarray*} 
0
&=& \int_\Omega \, \sum_{i,j=1}^n f_{\xi_i \xi_j}(x, Du) \, u_{x_j x_s} \, 2 \, \eta \, \eta_{x_i} \, (1 + |Du|^2)^\gamma \, u_{x_s} \, dx \\ 
&&+ \int_\Omega \, \sum_{i,j=1}^n f_{\xi_i \xi_j}(x, Du) \, u_{x_j x_s} \, \eta^2 \, (1 + |Du|^2)^\gamma \, u_{x_s x_i} \, dx \\
&&+ \int_\Omega \, \sum_{i,j=1}^n f_{\xi_i \xi_j}(x, Du) \, u_{x_j x_s} \, 2 \, \eta^2 \, \gamma \, (1 + |Du|^2)^{\gamma-1} \, |Du| \, D_{x_i}(|Du|) \, u_{x_s} \, dx \\
&&+ \int_\Omega \, \sum_{i=1}^n f_{\xi_i x_s}(x, Du) \, 2 \, \eta \, \eta_{x_i} \, (1 + |Du|^2)^\gamma \, u_{x_s} \, dx \\
&&+ \int_\Omega \, \sum_{i=1}^n f_{\xi_i x_s}(x, Du) \, \eta^2 \, (1 + |Du|^2)^\gamma \, u_{x_s x_i} \, dx \\
&&+ \int_\Omega \, \sum_{i=1}^n f_{\xi_i x_s}(x, Du) \, 2 \, \eta^2 \, \gamma \, (1 + |Du|^2)^{\gamma-1} \, |Du| \, D_{x_i}(|Du|) \, u_{x_s} \, dx \\
&&- \int_\Omega \, g \, 2 \, \eta \, \eta_{x_s} \, (1 + |Du|^2)^\gamma \, u_{x_s} \, dx \\
&&- \int_\Omega \, g \, 2 \, \eta^2 \, \gamma \, (1 + |Du|^2)^{\gamma - 1} \, |Du| \, D_{x_s}(|Du|) \, u_{x_s} \, dx \\
&&- \int_\Omega \, g \, \eta^2 \, (1 + |Du|^2)^\gamma \, u_{x_s x_s} \, dx \\
&:=& I_{1,s} + I_{2,s} + I_{3,s} + I_{4,s} + I_{5,s} + I_{6,s} + I_{7,s} + I_{8,s} + I_{9,s}
\end{eqnarray*}
We sum in the previous equation all terms with respect to $s$ from $1$ to $n$, and we denote by $I_1-I_9$ the corresponding integrals. We can estimate them following the lines of \cite{BR1}. 
\\
Summing up the nine terms and using \eqref{HpF2}, we obtain
\begin{eqnarray}
\label{final_eq}
&&\int_\Omega \eta^2 \, (1 + |Du|^2)^{\frac{p-2}{2} + \gamma} \, |D^2 u|^2 \, dx \nonumber \\
&\le& C \, \Theta \, (1 + \gamma^2) \left[ \int_\Omega (\eta^{2m} + |D \eta|^{2m}) \, (1 + |Du|^2)^{\left( q-\frac{p}{2}+\gamma \right) m} \, dx \right]^\frac{1}{m}
\end{eqnarray}
where the constant $C$ depends on $\nu, L, n, p, q$ but is independent of $\gamma$ and where we set
\\
\begin{equation*}
\Theta = 1 + \|g\|^2_{L^r(\Omega)} + \|h\|^2_{L^r(\Omega)}
\end{equation*}
and
\begin{equation}
\label{emme}
m = \frac{r}{r-2}
\end{equation}
By Sobolev Embedding Theorem, recalling that $p \ge 2$, we have
\begin{eqnarray*}
&& \left[ \int_\Omega \eta^{2^*} \, (1 + |Du|^2)^{\left( \gamma + \frac{p}{2} \right) \frac{2^*}{2}} \, dx \right]^\frac{2}{2^*} \\
&=& \left[ \int_\Omega \eta^{2^*} \, (1 + |Du|^2)^{\left( \frac{\gamma}{2} + \frac{p}{4} \right) 2^*} \, dx \right]^\frac{2}{2^*} \\
&\le& C \int_\Omega \left| D \left[ \eta \, (1 + |Du|^2)^{\frac{\gamma}{2}+\frac{p}{4}} \right] \right|^2 dx \\
&\le& C \int_\Omega |D\eta|^2 \, (1 + |Du|^2)^{\gamma+\frac{p}{2}} \, dx \\
&&+ \, C \, (1+\gamma^2) \int_\Omega \eta^2 \left[ (1 + |Du|^2)^{\frac{\gamma}{2}+\frac{p}{4}-1} \, |Du| \, |D^2u| \right]^2 \, dx \\
&=& C \int_\Omega |D\eta|^2 \, (1 + |Du|^2)^{\gamma+\frac{p}{2}} \, dx \\
&&+ \, C \, (1+\gamma^2) \int_\Omega \eta^2 \, (1 + |Du|^2)^{\gamma+\frac{p}{2}-2} \, |Du|^2 \, |D^2u|^2 \, dx \\
&=& C \int_\Omega |D\eta|^2 \, (1 + |Du|^2)^{\gamma+\frac{p}{2}} \, dx \\
&&+ \, C \, (1+\gamma^2) \int_\Omega \eta^2 \, (1 + |Du|^2)^{\gamma+\frac{p-2}{2}-1} \, |Du|^2 \, |D^2u|^2 \, dx \\
&\le& C \int_\Omega |D\eta|^2 \, (1 + |Du|^2)^{\gamma+\frac{p}{2}} \, dx \\
&&+ \, C \, (1+\gamma^2) \int_\Omega \eta^2 \, (1 + |Du|^2)^{\gamma+\frac{p-2}{2}} \, |D^2u|^2 \, dx
\end{eqnarray*}
where we set
\begin{equation*}
2^* =
	\begin{cases}
	\displaystyle \frac{2 \, n}{n-2} \, & \text{if $n \ge 3$} \\ \\
	\text{any finite exponent} \, & \text{if $n=2$}
	\end{cases}
\end{equation*}
and we can observe that $2^*>2$. In the case $n = 2$ we assume that $2^* > 2 \, m$.
Thanks to the left hand side of \eqref{coeff_eq}, we know that
\begin{equation*}
1 \le \frac{q}{p}
\end{equation*}
\begin{equation*}
p \le q
\end{equation*}
\begin{equation*}
\frac{p}{2} \le q - \frac{p}{2}
\end{equation*}
which allow us to say that
\begin{equation*}
\gamma + \frac{p}{2} \le q - \frac{p}{2} + \gamma
\end{equation*}
so we can write
\begin{eqnarray*}
\left[ \int_\Omega \eta^{2^*} \, (1 + |Du|^2)^{\left( \gamma + \frac{p}{2} \right) \frac{2^*}{2}} \, dx \right]^\frac{2}{2^*}
&\le& C \int_\Omega |D\eta|^2 \, (1 + |Du|^2)^{q-\frac{p}{2}+\gamma} \, dx \\
&&+ \, C \, (1+\gamma^2) \int_\Omega \eta^2 \, (1 + |Du|^2)^{\gamma+\frac{p-2}{2}} \, |D^2u|^2 \, dx
\end{eqnarray*}
Thanks to \eqref{final_eq}, we finally get
\begin{eqnarray*}
&&\left[ \int_\Omega \eta^{2^*} \, (1 + |Du|^2)^{\left( \gamma + \frac{p}{2} \right ) \frac{2^*}{2}} \, dx \right]^\frac{2}{2^*} \\
&\le& C \, \Theta \, (1 + \gamma^2) \left[ \int_\Omega (\eta^{2m} + |D \eta|^{2m}) \, (1 + |Du|^2)^{\left( q - \frac{p}{2}+\gamma \right) m} \, dx \right]^\frac{1}{m}
\end{eqnarray*}
from which we deduce
\begin{equation*}
\left[ \int_{B_\rho} (1 + |Du|^2)^{\left( \gamma + \frac{p}{2} \right) \frac{2^*}{2}} \, dx \right]^\frac{2}{2^*} \le \, C \, \frac{\Theta \, (1 + \gamma^2)}{(R - \rho)^2} \left[ \int_{B_R} (1 + |Du|^2)^{\left( q - \frac{p}{2}+\gamma \right) m} \, dx \right]^\frac{1}{m}
\end{equation*}
for any $0<\rho<R$. \\
At this point we introduce the quantity $\sigma$ defined as
\begin{equation}
\label{sigma}
\sigma = q - \frac{p}{2} - \frac{p}{2 \, m}
\end{equation}
where we observe that $\sigma >0$ due to left hand side inequality of assumption \eqref{coeff_eq} and the fact that $m>1$, in fact
\begin{equation*}
\sigma = q - \frac{p}{2} - \frac{p}{2 \, m} \ge p - \frac{p}{2} - \frac{p}{2 \, m} = \frac{p}{2} - \frac{p}{2 \, m} > 0
\end{equation*}
That allow us to say that
\begin{equation*}
q - \frac{p}{2} = \sigma + \frac{p}{2m}
\end{equation*}
Therefore
\begin{eqnarray*}
\left[ \int_{B_\rho} (1 + |Du|^2)^{\left( \gamma + \frac{p}{2} \right) \frac{2^*}{2}} \, dx \right ]^\frac{2}{2^*}
&\le& C \, \frac{\Theta \, (1 + \gamma^2)}{(R - \rho)^2} \left[ \int_{B_R} (1 + |Du|^2)^{\left( q - \frac{p}{2}+\gamma \right) m} \, dx \right]^\frac{1}{m} \\
&=& C \, \frac{\Theta \, (1 + \gamma^2)}{(R - \rho)^2} \left[ \int_{B_R} (1 + |Du|^2)^{\left( \sigma + \frac{p}{2m}+\gamma \right) m} \, dx \right]^\frac{1}{m}
\end{eqnarray*}
which allow us to say that
\begin{eqnarray}
\label{finale3}
&&\left[ \int_{B_\rho} (1 + |Du|^2)^{\left( \gamma + \frac{p}{2} \right) \frac{2^*}{2}} \, dx \right]^\frac{2}{2^*} \nonumber \\
&\le& C \, \frac{\Theta \, (1 + \gamma^2)}{(R - \rho)^2} \, \|1 + |Du|^2\|^\sigma_{L^\infty(B_R)} \left[ \int_{B_R} (1 + |Du|^2)^{\gamma \, m + \frac{p}{2}} \, dx \right]^\frac{1}{m}
\end{eqnarray}
\\
We now inductively define the exponents
\begin{eqnarray}
&& \gamma_1 := 0 \nonumber \\[2.0mm]
&& \gamma_{k+1} := \frac{1}{m} \left[ \left( \gamma_k + \frac{p}{2} \right) \frac{2^*}{2} - \frac{p}{2} \right] \label{tre_exps} \\[2.0mm]
&& \alpha_k := m \, \gamma_k + \frac{p}{2} \nonumber
\end{eqnarray}
for every integer $k \ge 1$. It follows that
\begin{equation}
\label{alpha}
\alpha_{k+1} = \left( \gamma_k + \frac{p}{2}\right) \frac{2^*}{2} \, = \, \chi \, \alpha_k + \tau
\end{equation}
where we have set
\begin{equation}
\label{chi_tau}
\chi := \frac{2^*}{2 \, m} \qquad \text{and} \qquad \tau := \frac{2^* \, \alpha_1}{r} = \frac{2^* \, p}{2 \, r}
\end{equation}
By induction we can prove that
\begin{equation}
\label{alpha_sum}
\alpha_{k+1} = \alpha_1 \, \chi^k + \tau \, \sum_{i=0}^{k-1} \chi^i = \frac{p}{2} \, \chi^k + \tau \, \sum_{i=0}^{k-1}\chi^i
\end{equation}
and
\begin{equation}
\label{gamma}
\gamma_{k+1} = \frac{\alpha_1}{m} \, (\chi^k - 1) + \frac{\tau}{m} \, \sum_{i=0}^{k-1} \chi^i = \frac{p}{2 \, m} \, (\chi^k - 1) + \frac{\tau}{m} \, \sum_{i=0}^{k-1} \chi^i 
\end{equation}
\\
Now we consider $0 < \rho_0 < R_0$ and set
\begin{equation*}
R_k = \rho_0 + \frac{R_0 - \rho_0}{2^k} \qquad \forall k \ge 1
\end{equation*}
so that $R_{k+1} \le \, R_k$ for all $k \ge 1$ and
\begin{equation*}
R_k - R_{k+1} = \frac{R_0 - \rho_0}{2^{k+1}}.
\end{equation*}
We rewrite \eqref{finale3} with $\rho = R_{k+1}$ and $R = R_k$. We obtain
\begin{eqnarray*}
&&\left[ \int_{B_{R_{k+1}}} (1 + |Du|^2)^{\left( \gamma_k + \frac{p}{2} \right) \frac{2^*}{2}} \, dx \right] ^\frac{2}{2^*} \\
&\le& C \, \frac{\Theta \, (1 + \gamma_k^2)}{(R_k - R_{k+1})^2} \, \|1 + |Du|^2\|^\sigma_{L^\infty(B_{R_k})} \left[ \int_{B_{R_k}} (1 + |Du|^2)^{m \, \gamma_k +\frac{p}{2}}  \, dx \right]^\frac{1}{m}
\end{eqnarray*}
from which we deduce
\begin{eqnarray*}
&& \left[ \int_{B_{R_{k+1}}} (1 + |Du|^2)^{\left( \gamma_k + \frac{p}{2} \right) \frac{2^*}{2}} \, dx \right] ^\frac{2}{2^*} \\
&\le& C \, \frac{4^{k+1} \, \Theta \, (1 + \gamma_k^2)}{(R_0 - \rho_0)^2} \, \|1 + |Du|^2\|^\sigma_{L^\infty(B_{R_k})} \left[ \int_{B_{R_k}} (1 + |Du|^2)^{m \, \gamma_k + \frac{p}{2}} \, dx \right]^\frac{1}{m}
\end{eqnarray*}
and we can write
\begin{eqnarray}
&& \int_{B_{R_{k+1}}} (1 + |Du|^2)^{\left( \gamma_k + \frac{p}{2} \right) \frac{2^*}{2}} \, dx \label{finale4} \\
&\le& \left[ C \, \frac{4^{k+1} \, \Theta \, (1 + \gamma_k^2)}{(R_0 - \rho_0)^2} \right]^\frac{2^*}{2} \, \|1 + |Du|^2\|^\frac{2^* \, \sigma}{2}_{L^\infty(B_{R_k})} \left[ \int_{B_{R_k}} (1 + |Du|^2)^{m \, \gamma_k + \frac{p}{2}} \, dx \right]^\chi \nonumber
\end{eqnarray}
\\
For each $k \in \N$, we define:
\begin{equation}
\label{Akappa}
A_k := \left[ \int_{B_{R_k}} (1+|Du|^2)^{\alpha_k} \, dx \right]^\frac{1}{\alpha_k}
\end{equation}
where we can use the definitions \eqref{tre_exps} and \eqref{alpha} for $\alpha_k$ and $\alpha_{k+1}$. So we can rewrite \eqref{finale4} as \\
\begin{equation*}
A_{k+1} \, \le \, \left[ C \, \frac{4^{k+1} \, \Theta \, (1+\gamma_k^2)}{(R_0 - \rho_0)^2} \right]^\frac{2^*}{2 \, \alpha_{k+1}} \, \left[ \|1 + |Du|^2\|_{L^\infty(B_{R_k})}^\frac{2^* \, \sigma }{2} \right]^\frac{1}{\alpha_{k+1}} \, A_k^\frac{\alpha_k \, \chi}{\alpha_{k+1}}
\end{equation*}
Iterating this inequality we obtain
\begin{equation}
\label{estimate}
A_{k+1} \, \le \, \prod_{i=1}^k \left[ C \, \frac{4^{i+1} \, \Theta \, (1+\gamma_i^2)}{(R_0 - \rho_0)^2} \right]^\frac{2^* \, \chi^{k- i}}{2 \, \alpha_{k+1}} \, \left[ \|1 + |Du|^2\|_{L^\infty(B_R)}^\frac{2^* \, \sigma }{2} \right]^{\frac{1}{\alpha_{k+1}} \, \sum_{i=0}^{k-1} \chi^i} \, A_1^\frac{\chi^k \, \alpha_1}{\alpha_{k+1}}
\end{equation}
\\
Thanks to \eqref{alpha_sum}, it's easy to see that
\begin{equation}
\label{exp_norma}
\lim_{k \rightarrow +\infty} \, \frac{1}{\alpha_{k+1}} \, \sum_{i = 0}^{k-1} \, \chi^i \, = \, \frac{1}{\alpha_1 \, (\chi - 1) + \tau} 
\end{equation}
and that
\begin{equation}
\label{delta}
\lim_{k \rightarrow +\infty} \, \frac{\chi^k \, \alpha_1}{\alpha_{k+1}} \, = \, \frac{\alpha_1 \, (\chi - 1)}{\alpha_1 \, (\chi - 1) + \tau} = \delta 
\end{equation}
\\
We define:
\begin{eqnarray*}
M_k
&:=& \prod_{i=1}^k \, \left[ C \, 4^{i+1} \, (1 + \gamma_i^2) \right]^\frac{2^* \, \chi^{k-i}}{2 \, \alpha_{k+1}} \\
&=& \exp\left\{ \frac{2^*}{2 \, \alpha_{k+1}} \, \sum_{i=1}^k \, \chi^{k - i} \, \log\left[ C \, 4^{i+1} \, (1+\gamma_i^2)\right] \right\}
\end{eqnarray*}
\\
and it isn't difficult to see, thanks to \eqref{alpha_sum}, that
\begin{equation}
\label{emme_kappa}
\lim_{k \rightarrow +\infty} M_k \, \le \, M \, < \, +\infty
\end{equation}
\\
Last but not least, as proved in \cite{BR1}, we can say that
\begin{equation}
\label{Ics}
X \, := \, \lim_{k \rightarrow +\infty} \sum_{i=1}^k \, \frac{2^* \, \chi^{k- i}}{2 \, \alpha_{k+1}} \, < \, \infty
\end{equation}
Thanks to \eqref{exp_norma}, \eqref{delta}, \eqref{emme_kappa} and \eqref{Ics}, letting $k \rightarrow +\infty$, noticing that $R_0 > R_1$, we can rewrite \eqref{estimate} as follows
\begin{equation*}
\|1 + |Du|^2\|_{L^\infty(B_{\rho_0})} \, \le \, \frac{M \, \Theta^X}{(R_0 - \rho_0)^{2 \, X}} \, \|1+|Du|^2\|^\frac{2^* \, \sigma}{2 \, \alpha_1 \, (\chi - 1) + 2 \, \tau} \, \left[ \int_{B_{R_0}} (1+|Du|^2)^\frac{p}{2} \, dx \right]^\frac{2 \, \delta}{p}
\end{equation*}
\\
Now we have that
\begin{eqnarray}
E
&:=& \frac{2^* \, \sigma}{2 \, \alpha_1 \, (\chi - 1) + 2 \, \tau} \nonumber \\
&=& \frac{\frac{2^* \, \sigma}{2}}{\frac{p}{2} \, (\chi - 1) + \frac{2^* \, \alpha_1}{r}} \nonumber \\
&=& \frac{\frac{2^*}{2} \, \left[ q - p \, \left( \frac{1}{2} + \frac{1}{2 \, m} \right) \right]}{\frac{p}{2} \, \left( \frac{2^*}{2 \, m} - 1 \right) + \frac{2^* \, p}{2 \, r}} \nonumber \\
&=& \frac{2^* \, \left[ q - p \, \left( \frac{1}{2} + \frac{1}{2 \, m} \right) \right]}{p \, \left[ \frac{2^*}{2 \, m} - 1 + \frac{2^*}{r} \right]} \label{exp_E}
\end{eqnarray}
\\
and we have to prove that $E<1$. Now, if $n \ge 3$ then we are done if and only if we have that
\begin{equation*}
2^* \, \left[ q - p \, \left( \frac{1}{2} + \frac{1}{2 \, m} \right) \right] \, < \, p \, \left[ \frac{2^*}{2 \, m} - 1 + \frac{2^*}{r} \right]
\end{equation*}
\begin{equation*}
2^* \, \left[ q - p \, \left( \frac{1}{2} + \frac{1}{2 \, m} \right) \right] \, < \, 2^* \, p \, \left[ \frac{1}{2 \, m} - \frac{1}{2^*} + \frac{1}{r} \right]
\end{equation*}
\begin{equation*}
q \, < \, p \, \left[ \frac{1}{2} + \frac{1}{2 \, m} + \frac{1}{2 \, m} - \frac{1}{2^*} + \frac{1}{r} \right]
\end{equation*}
but using the equality 
\\
\begin{equation}
\label{identita_Eleuteri}
\frac{1}{2 \, m} - \frac{1}{2^*} \, = \, \frac{1}{n} - \frac{1}{r}
\end{equation}
we know that
\begin{eqnarray*}
\frac{1}{2} + \frac{1}{2 \, m} + \frac{1}{2 \, m} - \frac{1}{2^*} + \frac{1}{r}
&=& \frac{1}{2} + \frac{2}{2^*} + \frac{2}{n} - \frac{2}{r} - \frac{1}{2^*} + \frac{1}{r} \\
&=& \left[ \frac{1}{2} + \frac{1}{2^*} + \frac{2}{n} \right] - \frac{1}{r} \\
&=& \left[ \frac{1}{2} + \frac{n - 2}{2 \, n} + \frac{2}{n} \right] - \frac{1}{r} \\
&=& \frac{n + n - 2 + 4}{2 \, n} - \frac{1}{r} \\
&=& 1 + \frac{1}{n} - \frac{1}{r}
\end{eqnarray*}
\\
and, thanks to \eqref{coeff_eq}, the thesis is proved. On the other hand, if $n = 2$, then passing to the limit $2^* \rightarrow \infty$ in \eqref{exp_E} we deduce, thanks to \eqref{emme} and \eqref{sigma}, that
\begin{equation*}
\frac{\sigma}{p \, \left( \frac{1}{2 \, m} + \frac{1}{r} \right)} \, < \, 1
\end{equation*}
if and only if
\begin{equation*}
q - \frac{p}{2} - \frac{p}{2 \, m} \, < \, p \, \left( \frac{1}{2 \, m} + \frac{1}{r} \right)
\end{equation*}
\begin{equation*}
q \, < \, p \, \left[ \frac{1}{2} + \frac{1}{m} + \frac{1}{r} \right]
\end{equation*}
\begin{equation*}
q \, < \, p \, \left[ \frac{1}{2} + \frac{r - 2}{r} + \frac{1}{r} \right]
\end{equation*}
\begin{equation*}
q \, < \, p \, \left [\frac{3}{2} - \frac{1}{r} \right]
\end{equation*}
which is nothing but \eqref{coeff_eq} with the choice $n = 2$. Thus we can use the Young's inequality with exponents
\begin{equation*}
\frac{2 \, \alpha_1 \, (\chi - 1) + 2 \, \tau}{2^* \, \sigma} \qquad \textnormal{and} \qquad \frac{2 \, \alpha_1 \, (\chi - 1) + 2 \, \tau}{[2 \, \alpha_1 \, (\chi - 1) + 2 \, \tau] - 2^* \, \sigma}
\end{equation*}
to get:
\begin{equation*}
\|1 + |Du|^2\|_{L^\infty(B_{\rho_0})} \, \le \, \frac{1}{2} \, \|1 + |Du|^2\|_{L^\infty(B_{R_0})} + \left[ \frac{M \, \Theta^X}{(R_0 - \rho_0)^{2 \, X}} \right]^\vartheta \, \left[ \int_{B_{R_0}} (1+|Du|^2)^\frac{p}{2} \, dx  \right]^\frac{2 \, \delta \, \vartheta}{p}
\end{equation*}
for an exponent $\vartheta = \vartheta\,(n, r, p, q) > 0$. \\ \\
Since the previous estimate holds true for $\rho < \rho_0 < R_0 < R$, once more by Lemma \ref{holf} we finally get
\begin{equation*}
\|1 + |Du|^2\|_{L^\infty(B_\rho)} \, \le \, \left[ \frac{M \, \Theta^X}{(R - \rho)^{2 \, X}} \right]^\vartheta \, \left[ \int_{B_R} (1+|Du|^2)^\frac{p}{2} \, dx \right]^\frac{2 \, \delta \, \vartheta}{p}
\end{equation*}
so we can write, thanks to \eqref{HpF1}
\begin{eqnarray*}
\|Du\|_{L^\infty(B_\rho)}
&\le& \left[ \frac{M \, \Theta^X}{(R - \rho)^{2 \, X}} \right]^\vartheta \, \left[ \int_{B_R} (1+|Du|^2)^\frac{p}{2} \, dx \right]^\frac{2 \, \delta \, \vartheta}{p} \\
&\le& \left[ \frac{M \, \Theta^X}{(R - \rho)^{2 \, X}} \right]^\vartheta \, \left[ \int_{B_R} \frac{1}{\nu} \, f(x, Du) \, dx \right]^\frac{2 \, \delta \, \vartheta}{p} \\
&\le& C \, \left\{ \int_{B_R} \left[1 + f(x, Du) \right] \, dx \right\}^\beta
\end{eqnarray*}
\begin{rmk*}
We can rewrite the equation \eqref{a-priori} in a different way as follows
\begin{equation}
\label{a-priori-approx}
\|Du\|_{L^\infty(B_\rho)} \, \le \, C \, \left[ \|1 + h\|_{L^r(B_R)} \right]^{\alpha \, \gamma} \, \left\{ \int_{B_R} \left[1 + f(x, Du) \right] \, dx \right\}^\frac{\gamma}{p}
\end{equation}
\end{rmk*}


\section{Approximation in case of Occurrence of Lavrentiev Phenomenon and proof of Theorem \ref{lip-solution}}

\label{approx-lavr}

\subsection{Approximation procedure}

We start by proving a lemma which will be necessary in the following. The proof of this lemma follows similarly as in \cite{BM76}, see also \cite{EMM3}, where the same result has been achieved for functionals. Here the novelty is the application to obstacle problems.
\begin{lemma}
\label{conv-funz}
For each $u \in \mathcal{K}_\psi(\Omega)$, there exists a sequence $u_k \in \mathcal{K}^*_\psi(\Omega)$ such that $u_k \rightharpoonup u$ weakly in $W^{1,p}(\Omega)$ and
\begin{equation*}
\overline{\mathcal{F}}(u) \, = \, \lim_{k \rightarrow +\infty} \mathcal{F}(u_k)
\end{equation*}
\end{lemma}
\begin{proof}
Let $u \in \mathcal{K}_\psi(\Omega)$ such that $\overline{\mathcal{F}}(u) < \infty$. Then, for all $k$, there exists $u_h^{(k)} \in \mathcal{K}^*_\psi(\Omega)$ such that $u_h^{(k)} \rightharpoonup u$, as $h \rightarrow +\infty$, weakly in $W^{1,p}(\Omega)$ and
\begin{equation*}
\overline{\mathcal{F}}(u) \, \le \, \lim_{h \rightarrow +\infty} \int_\Omega f(x, Du_h^{(k)}) \, dx \, \le \, \overline{\mathcal{F}}(u) + \frac{1}{k}
\end{equation*}
Moreover, by the weak convergence of $u_h^{(k)}$ in $W^{1,p}(\Omega)$ we get
\begin{equation*}
\lim_{h \rightarrow +\infty} \|u_h^{(k)} - u\|_{L^p(\Omega)} \, = \, 0
\end{equation*}
and, for $h$ sufficiently large,
\begin{equation*}
\int_\Omega |Du_h^{(k)}|^p \, dx \, \le \, \int_\Omega f(x, Du_h^{(k)}) \, dx \, \le \, \overline{\mathcal{F}}(u) + 1
\end{equation*}
Then for all $k$ there exists $h_k$ such that for all $h \ge h_k$
\begin{equation*}
\|u_h^{(k)} - u\|_{L^p(\Omega)} \, < \, \frac{1}{k}
\end{equation*}
and for $h = h_k$, by denoting $w_k = u_{h_k}^{(k)}$, we have
\begin{equation*}
\|w_k - u\|_{L^p(\Omega)} \, < \, \frac{1}{k} \qquad \text{and} \qquad \int_\Omega |Dw_k|^p \, dx \, \le \, C
\end{equation*}
then $w_k \rightharpoonup u$, as $k \rightarrow +\infty$, in the weak topology of $W^{1,p}(\Omega)$ and
\begin{equation*}
\overline{\mathcal{F}}(u) \, \le \, \int_\Omega f(x, Dw_k) \, dx \, \le \, \overline{\mathcal{F}}(u) + \frac{1}{k}
\end{equation*}
i.e.
\begin{equation*}
\lim_{k \rightarrow +\infty} \int_\Omega f(x, Dw_k) \, dx \, = \, \overline{\mathcal{F}}(u)
\end{equation*}
\end{proof}

For the approximation we are gonna now consider the situation where the Lavrentiev phenomenon may occur. We are gonna use the following Theorem \ref{funz-approx} in order to rewrite $f$ through a suitable sequence of regular functions: that, together with Theorem \ref{lip-solution}, will permit us to prove that the supplementary assumption $u \in W^{1,\infty}_\loc(\Omega)$ we supposed in Section \ref{a-priori} can be removed.

\begin{theorem}
\label{funz-approx}
Let $f$ be satisfying the growth conditions \eqref{HpF1}, $f_{\xi \xi}$ and $f_{\xi x}$ be two Carath\'eodory functions, satisfying \eqref{HpF2} and \eqref{HpF3} and $f$ strictly convex at infinity. Then there exists a sequence of $\mathcal{C}^2-$functions $f^{l k} : \Omega \times \R^n \rightarrow [0, +\infty)$, $f^{l k}$ convex in the last variable and strictly convex at infinity, such that $f^{l k}$ converges to $f$ as $l \rightarrow \infty$ and $k \rightarrow \infty$ for a.e. $x \in \Omega$, for all $\xi \in \R^n$ and uniformly in $\Omega_0 \times K$, where $\Omega_0 \Subset \Omega$ and $K$ being a compact set of $\R^n$. Moreover the functions $f^{l k}$ satisfy the hypothesis \eqref{HpF1}, \eqref{HpF2}, \eqref{HpF3} with constants which are independent on $k$ and satisfy the additional hypothesis necessaries to conclude our proof with constants which are dependent only on $k$.
\end{theorem}
\begin{proof}
We argue as in \cite{EMM16_b}, \cite{EMM3} and \cite{FFM02}. For the sake of completeness, we give a sketch of the arguments of the proof.
\\
Let $B$ be the unit ball of $\R^n$ centered in the origin and consider a positive decreasing sequence $\varepsilon_l \rightarrow 0$. We introduce
\begin{equation*}
f^l(x, \xi) \, = \, \int_{B \times B} \rho(y) \, \rho(\eta) \, f(x + \varepsilon_l \, y, \xi + \varepsilon_l \, \eta) \, d\eta \, dy
\end{equation*}
where $\rho$ is a positive symmetric mollifier and
\begin{equation}
\label{f-lk}
f^{l k}(x, \xi) \, = \, f^l(x, \xi) + \frac{1}{k} \, (1 + |\xi|^2)^\frac{q}{2}
\end{equation}
It is easy to check that the sequence $f^{l k}$ satisfies the hypothesis \eqref{HpF1}, \eqref{HpF2}, \eqref{HpF3} with constants which are independent on $k$ and satisfies the additional hypothesis necessaries to conclude our proof with constants which are dependent only on $k$. 
\end{proof}

\subsection{Proof of Theorem \ref{lip-solution}}

\noindent
Now we are ready to give the proof of Theorem \ref{lip-solution}. 

\begin{proof}
For $u_0 \in W^{1, q}(\Omega)$, let us consider the variational problems
\begin{equation}
\label{inf-f-lk}
\inf \left\{ \int_\Omega f^{l k}(x, Du) \, dx : u \in \mathcal{K}^*_\psi(\Omega) \right\}
\end{equation}
where $ f^{l k}$ are defined in \eqref{f-lk}. By semicontinuity arguments and direct methods of the Calculus of Variations, there exists $u^{l k} \in \mathcal{K}^*_\psi(\Omega)$, a solution to \eqref{inf-f-lk}. By the growth conditions and the minimality of $u^{l k}$, remembering that $u_0 \in \mathcal{K}^*_\psi(\Omega)$ because it is not restrictive to assume $u_0 \ge \psi$ as observed in Remark \ref{oss-classe-vuota}, we get
\begin{eqnarray*}
\int_\Omega |Du^{l k}|^p \, dx
&\le& \int_\Omega f^{l k}(x, Du^{l k}) \, dx \\
&\le& \int_\Omega f^{l k}(x, Du_0) \, dx \\
&=& \int_\Omega f^l(x, Du_0) \, dx + \frac{1}{k} \int_\Omega (1 + |Du_0|^2)^\frac{q}{2} \, dx
\end{eqnarray*}
Moreover, the properties of the convolutions imply that
\begin{equation*}
f^l(x, Du_0) \, \overset{l \rightarrow \infty}{\longrightarrow} \, f(x, Du_0) \qquad \text{a.e. in } \Omega
\end{equation*}
and since
\begin{equation*}
\int_\Omega f^l(x, Du_0) \, dx \, \le \, C \int_\Omega (1 + |Du_0|^2)^\frac{q}{2} \, dx
\end{equation*}
where is fundamental the hypothesis that $u_0 \in W^{1, q}(\Omega)$. By Lebesgue Dominated Covergence Theorem we deduce therefore
\begin{eqnarray*}
\lim_{l \rightarrow +\infty} \int_\Omega |Du^{l k}|^p \, dx
&\le& \lim_{l \rightarrow +\infty} \int_\Omega f^l(x, Du_0) \, dx + \frac{1}{k} \int_\Omega (1 + |Du_0|^2)^\frac{q}{2} \, dx \\
&=& \int_\Omega f(x, Du_0) \, dx + \frac{1}{k} \int_\Omega (1 + |Du_0|^2)^\frac{q}{2} \, dx
\end{eqnarray*}
By Theorem \ref{funz-approx}, the functions $f^{l k}$ satisfy \eqref{HpF1}, \eqref{HpF2} and \eqref{HpF3}, so we can apply the a-priori estimate \eqref{a-priori-approx} to $u^{l k}$ and obtain, by standard covering arguments for all $B \Subset \Omega$,
\begin{equation*}
\|Du^{l k}\|_{L^\infty(B)} \, \le \, C \, [\|1 + h_l\|_{L^r(\Omega)}]^{\alpha \, \gamma} \, \left\{ \int_\Omega [1 + f^{l k}(x, Du^{l k})] \, dx \right\}^\frac{\gamma}{p}
\end{equation*}
Since $\|1 + h_l\|_{L^r(\Omega)} = \|(1 + h)_l\|_{L^r(\Omega)} \le \|1 + h\|_{L^r(\Omega)}$, we obtain
\begin{eqnarray*}
\|Du^{l k}\|_{L^\infty(B)}
&\le& C \, [\|1 + h\|_{L^r(\Omega)}]^{\alpha \, \gamma} \, \left\{ \int_\Omega [1 + f^{l k}(x, Du^{l k})] \, dx \right\}^\frac{\gamma}{p} \\
&\le& C \, [\|1 + h\|_{L^r(\Omega)}]^{\alpha \, \gamma} \, \left\{ \int_\Omega [1 + f^l(x, Du_0) + \frac{1}{k} \, (1 + |Du_0|^2)^\frac{q}{2}] \, dx \right\}^\frac{\gamma}{p} \\
&\le& C \, [\|1 + h\|_{L^r(\Omega)}]^{\alpha \, \gamma} \, \left\{ \int_\Omega [1 + f(x, Du_0) + \frac{1}{k} \, (1 + |Du_0|^2)^\frac{q}{2}] \, dx \right\}^\frac{\gamma}{p}
\end{eqnarray*}
where $C, \gamma, \beta$ depend on $n, r, p, q, L, \nu, B, \Omega$ and on the local bound for $\|D \psi\|_{W^{1, r}}$ but are independent of $l, k$. Therefore we conclude that there exist $u^k \in \mathcal{K}_\psi(\Omega)$, for all $k \in \N$, such that
\begin{equation*}
u^{l k} \, \overset{l \rightarrow \infty}{\longrightarrow} \, u^k \, \text{ weakly in } W^{1, p}(\Omega)
\end{equation*}
\begin{equation*}
u^{l k} \, \overset{l \rightarrow \infty}{\longrightarrow} \, u^k \, \text{ weakly star in } W^{1, \infty}_\loc(\Omega)
\end{equation*}
and by the previous estimates
\begin{eqnarray*}
\|Du^k\|_{L^p(\Omega)}
&\le& \liminf_{l \rightarrow \infty} \|Du^{l k}\|_{L^p(\Omega)} \\
&\le& \int_\Omega f(x, Du_0) \, dx + \int_\Omega (1 + |Du_0|^2)^\frac{q}{2} \, dx
\end{eqnarray*}
and
\begin{eqnarray*}
\|Du^k\|_{L^\infty(B)}
&\le& \liminf_{l \rightarrow \infty} \|Du^{l k}\|_{L^\infty(B)} \\
&\le& C \, [\|1 + h\|_{L^r(\Omega)}]^{\alpha \, \gamma} \, \left\{ \int_\Omega [1 + f(x, Du_0) + \frac{1}{k} \, (1 + |Du_0|^2)^\frac{q}{2}] \, dx \right\}^\frac{\gamma}{p}
\end{eqnarray*}
Thus we can deduce that there exists, up to subsequences, $\overline{u} \in \mathcal{K}_\psi(\Omega)$, thanks to the fact that $\mathcal{K}_\psi(\Omega)$ is convex and close, such that
\begin{equation*}
u^k \rightarrow \overline{u} \, \text{ weakly in } W^{1, p}(\Omega)
\end{equation*}
\begin{equation*}
u^k \rightarrow \overline{u} \, \text{ weakly star in } W^{1, \infty}_\loc(\Omega)
\end{equation*}
Now, we have that the functional is lower semicontinuous in $W^{1,1}(\Omega)$ and $u^k \rightarrow \overline{u}$ weakly in $W^{1, p}(\Omega)$, so $u^k \rightarrow \overline{u}$ weakly in $W^{1,1}(\Omega)$ too.
\\
Hence, thanks to \eqref{vupi}, by compact embedding and noticing that
\begin{equation*}
\int_B |D(V_p(Du_\varepsilon(x)))|^2 \, dx \, \le \, C \, \left( \int_\Omega (1 + |Du_\varepsilon(x)|)^p \, dx \right)^\beta
\end{equation*}
with $C$ and $\beta$ depending on $n, r, p, q, \nu, L, R, \rho$, we infer
\begin{equation*}
V_p(Du^k) \rightharpoonup w \textnormal{ weakly in }W^{1,2}_{loc}(\Omega)
\end{equation*}
\begin{equation*}
V_p(Du^k) \rightarrow w \textnormal{ strongly in }L^{2}_{loc}(\Omega)
\end{equation*}
from which we deduce, together with inequality \eqref{stima_vupi},
\begin{equation*}
Du^k \rightarrow \overline{w} \textnormal{ strongly in } L^p_{loc}(\Omega)
\end{equation*}
We thus have the strong convergence 
\begin{equation*}
u^k \rightarrow \overline{u} \textnormal{ in } W^{1,p}_0(\Omega) + u_0
\end{equation*}
and the limit function $ \overline{u} $ still belongs to $K_{\psi}(\Omega)$, since this set is closed.
\noindent
\\
For any fixed $k \in \N$, using the uniform convergence of $f^l$ to $f$ in $\Omega_0 \times K$ (for any $K$ compact subset of $\R^n$) and the minimality of $u^{l k}$, we get for all $w \in \mathcal{K}^*_\psi(\Omega)$
\begin{eqnarray*}
\int_{\Omega_0} f(x, Du^k) \, dx
&\le& \liminf_{l \rightarrow \infty} \int_{\Omega_0} f(x, Du^{l k}) \, dx \\
&=& \liminf_{l \rightarrow \infty} \int_{\Omega_0} f^l(x, Du^{l k}) \, dx \\
&\le& \liminf_{l \rightarrow \infty} \int_{\Omega_0} f^l(x, Du^{l k}) \, dx + \frac{1}{k} \int_\Omega (1 + |Du^{l k}|^2)^\frac{q}{2} \, dx \\
&\le& \liminf_{l \rightarrow \infty} \int_\Omega f^l(x, Du^{l k}) \, dx + \frac{1}{k} \int_\Omega (1 + |Du^{l k}|^2)^\frac{q}{2} \, dx \\
&\le& \liminf_{l \rightarrow \infty} \int_\Omega f^l(x, Dw) \, dx + \frac{1}{k} \int_\Omega (1 + |Dw|^2)^\frac{q}{2} \, dx
\end{eqnarray*}
Then, for $\Omega_0 \rightarrow \Omega$, using Dominated Convergence Theorem, we have
\begin{equation*}
\int_\Omega f(x, Du^k) \, dx \, \le \, \int_\Omega f(x, Dw) \, dx + \frac{1}{k} \int_\Omega (1 + |Dw|^2)^\frac{q}{2} \, dx
\end{equation*}
By definition \eqref{relax}, we have
\begin{equation}
\label{min-relax}
\overline{\mathcal{F}}(\overline{u}) \, \le \, \liminf_{k \rightarrow \infty} \int_\Omega f(x, Du^k) \, dx \, \le \, \int_\Omega f(x, Dw) \, dx \qquad \forall \, w \in \mathcal{K}^*_\psi(\Omega)
\end{equation}
Let $v \in \mathcal{K}_\psi(\Omega)$. By Lemma \ref{conv-funz}, there exists $u_k \in \mathcal{K}^*_\psi(\Omega)$ such that $u_k \rightharpoonup v$ weakly in $W^{1, p}(\Omega)$ and
\begin{equation*}
\lim_{k \rightarrow \infty} \int_\Omega f(x, Du_k) \, dx \, = \, \overline{\mathcal{F}}(v)
\end{equation*}
By \eqref{min-relax},
\begin{equation*}
\overline{\mathcal{F}}(\overline{u}) \, \le \, \int_\Omega f(x, Du_k) \, dx
\end{equation*}
and we can conclude that
\begin{equation*}
\overline{\mathcal{F}}(\overline{u}) \, \le \, \lim_{k \rightarrow \infty} \int_\Omega f(x, Du_k) \, dx \, = \, \overline{\mathcal{F}}(v) \qquad \forall \, v \in \mathcal{K}_\psi(\Omega)
\end{equation*}
Then $\overline{u} \in W^{1, \infty}_\loc(\Omega)$ is a solution to the problem $\min \{ \overline{\mathcal{F}}(u) : u \in \mathcal{K}_\psi(\Omega) \}$.
\end{proof}


\end{document}